\documentclass{amsart}

\usepackage{amssymb}

\newtheorem{theorem}{Theorem}[section]
\newtheorem{corollary}[theorem]{Corollary}
\newtheorem{proposition}[theorem]{Proposition}

\newtheorem{lemma}[theorem]{Lemma}

\theoremstyle{definition}
\newtheorem{definition}[theorem]{Definition}

\newtheorem{remark}[theorem]{Remark}

\newcommand{\x}{\times}

\newcommand{\bd}{\partial}
\newcommand{\cB}{{\mathcal{B}}}

\newcommand{\cO}{\mathcal O}
 
\newcommand{\cV}{\mathcal{V}} 
 
\newcommand{\cU}{\mathcal{U}}

\newcommand{\ZZ}{\mathbb{Z}}
\newcommand{\CC}{\mathbb{C}}

\newcommand{\RR}{\mathbb{R}}

\newcommand{\QQ}{\mathbb{Q}}

\DeclareMathOperator{\Id}{Id}

\renewcommand{\a}{\alpha}
\renewcommand{\b}{\beta}

\newcommand{\e}{\varepsilon}

\newcommand{\f}{\varphi}

\renewcommand{\l}{\lambda}

\newcommand{\s}{\sigma}

\renewcommand{\o}{\omega}
\newcommand{\p}{\phi}

\newcommand{\G}{\Gamma}

\renewcommand{\S}{\Sigma}

\newcommand{\orb}{\mathrm{orb}}

\newcommand{\ii}{\mathrm{i}}

\begin{document}

\title[Constructions of symplectic surfaces in symplectic $4$-manifolds]
{Constructions of symplectic surfaces in symplectic $4$-manifolds with transversal intersections}

\author{Vicente Mu\~noz}
\address{Departamento de Algebra, Geometr\'{\i}a y Topolog\'{\i}a, Universidad Complutense de Madrid, Plaza de Ciencias, Ciudad Universitaria, 28040 Madrid, Spain} 
\email{vicente.munoz@ucm.es}

\author{Juan Rojo}
\address{ETS Ingenieros Inform\'aticos, Universidad Polit\'ecnica de Madrid,
Campus de Montegancedo, 28660 Madrid, Spain}
\email{juan.rojo.carulli@upm.es}

\maketitle

\begin{abstract}
In the breakthrough paper \cite{Mu-jems}, it 
is constructed the first example of a simply connected compact 
$5$-manifold (aka.\ Smale-Barden manifold) which admits a K-contact structure but does not carry a Sasakian structure, thus settling the question raised 
as Open Problem 10.2.1 in \cite{BG}. 

In this paper we 
revise, refine and generalize the constructions of symplectic surfaces in a symplectic $4$-manifold with transversal intersections. These are needed to produce the ramification locus of
Seifert bundles over symplectic $4$-orbifolds that serve to produce
K-contact $5$-manifolds. 
\end{abstract}

\section{Introduction}
\label{sec:1}

A central question in geometry is to determine when a given manifold admits a specific geometric structure. 
Complex geometry provides numerous examples of compact manifolds with rich topology, and there is a number
of topological properties that are satisfied by K\"ahler manifolds \cite{ABCKT, DGMS}. If we forget about the integrability of the complex structure, then we are dealing with symplectic manifolds. There has been enormous interest in the construction of (compact) symplectic manifolds that do not admit K\"ahler structures,
and in determining its topological properties \cite{OT}.

In odd dimensions, the analogues of K\"ahler and symplectic manifolds are Sasakian and K-contact manifolds. 
A K-contact manifold is a $(2n+1)$-dimensional manifold $M$ endowed with a contact form $\eta\in \Omega^1(M)$, i.e.\ $\eta\wedge (d\eta)^n>0$ everywhere, and with an endomorphism $\Phi$ of $TM$ such that:
 \begin{itemize}
\item $\Phi^2=-\Id + \xi\otimes\eta$, where $\xi$ is the Reeb
vector field of $\eta$ (that is $i_\xi \eta=1$, $i_\xi d\eta=0$),
\item 
$d\eta (\Phi X,\Phi Y)\,=\,d\eta (X,Y)$, for all vector fields $X,Y$,
\item $d\eta (\Phi X,X)>0$ for all nonzero $X\in \ker \eta$, and
\item the Reeb field $\xi$ is Killing with respect to the
Riemannian metric 
 $g(X,Y)=d\eta (\Phi X,Y)+\eta (X)\eta(Y)$.
\end{itemize}
A Sasakian manifold $M$ is a K-contact manifold for which 
the Nijenhuis
tensor $N_{\Phi} (X, Y) = {\Phi}^2 [X, Y] + [\Phi X, \Phi Y] -  \Phi [ \Phi X, Y] - \Phi [X, \Phi Y]$
satisfies $N_{\Phi} \,=\,-{d}\eta\otimes \xi$.

Sasakian geometry has become an important and active subject since the treatise of Boyer and Galicki \cite{BG}, 
and there is much interest on constructing K-contact manifolds which do not admit Sasakian structures.
In higher dimensions, the problem is solved by means of topological obstructions, like the parity of $b_1$. When one moves to the case of simply connected manifolds, K-contact non-Sasakian 
examples of dimensions $\geq 7$ are constructed by using the rational homotopy type or the hard Lefschetz property. 
The problem of the existence of simply connected K-contact non-Sasakian compact manifolds in dimension $5$ is the hardest and a cornerstone in this area. It appears as Open Problem 10.2.1 in \cite{BG}:

\smallskip

{\it Are there Smale-Barden manifolds with K-contact but not Sasakian structures?}

\medskip

A simply connected compact $5$-manifold is called a {\it Smale-Barden manifold}. These manifolds are classified \cite{B,S} by their second homology group 
   \begin{equation*} 
  H_2(M,\ZZ)=\ZZ^k\oplus ( {\textstyle\bigoplus\limits_{p,i}}\, \ZZ_{p^i}^{c(p^i)}),
  \end{equation*}
where $k=b_2(M)$ and we have written as a direct sum of
cyclic groups of prime power order, and by the second 
Stiefel-Whitney map $w_2: H_2(M,\ZZ)\rightarrow\ZZ_2$. The fact that we have a classification of all 
Smale-Barden manifolds raises the {\em geography problem} of determining which Smale-Barden manifolds admit K-contact and which admit Sasakian structures.

A Sasakian manifold $M$ always admits a \emph{quasi-regular} Sasakian structure \cite{R}, that is it is a 
Seifert bundle over a cyclic K\"ahler orbifold $X$.
In the case of a $5$-manifold, $X$ is a singular
complex surface with cyclic quotient singularities, and the isotropy locus is formed by complex curves intersecting transversely. 
Analogously, a K-contact manifold admits a quasi-regular K-contact structure \cite{MT}, which is a Seifert bundle over an
almost-K\"ahler orbifold
(that is, a symplectic orbifold with a compatible almost complex structure, which always exists) with cyclic singularities, and the isotropy locus is formed by symplectic surfaces.

The topology of a simply connected $5$-manifold $M$ which is a Seifert bundle over a cyclic $4$-orbifold $X$ is determined in \cite[Therorem 36]{Mu}. The homology is 
 $$
 H_2(M,\ZZ)=\ZZ^k\oplus ({\textstyle\bigoplus\limits_i}\, 
 \ZZ_{m_i}^{2g_i}),
 $$
where $k=b_2(M)=b_2(X)-1$, and the isotropy locus has 
isotropy surfaces $D_i$ with multiplicities $m_i$, and genus $g_i=g(D_i)$. We have the conditions $H_1(X,\ZZ)=0$,  
$\gcd(m_i,m_j)=1$ if $D_i,D_j$ intersect, 
and $H^2(X,\ZZ)\to \mathop{\oplus}_i H^2(D_i,\ZZ_{m_i})$ is surjective.
There are some technical conditions on the Chern class of the
Seifert bundle $M\to X$, and on the orbifold fundamental group 
$\pi_1^{\orb}(X)$ to arrange that $M$ is simply-connected, and to characterize the Stiefel-Whiteny class $w_2(M)$, but we shall not enter into this.

 To produce K-contact Smale-Barden manifolds, one needs to construct symplectic  $4$-orbifolds with cyclic quotient
 singularities having symplectic surfaces 
 of given genus inside. If the isotropy coefficients are not coprime, these surfaces are forced to be
 disjoint (and linearly independent in homology). In particular
 the number of such surfaces in the isotropy locus is at most
 $b_2(X)=k+1$. The genus of the isotropy surfaces, the isotropy 
 coefficients, and whether they are disjoint, are read in 
 $H_2(M,\ZZ)$ by the above.
 
The question \cite[Open Problem 10.2.1]{BG} is solved in the breakthrough paper \cite{Mu-jems}:

\begin{theorem}[{\cite{Mu-jems}}] \label{thm:main}
There exists a Smale-Barden manifold $M$ which admits a K-contact structure but does not admit a Sasakian structure.
In concrete, there is some $N>0$ large enough, and distinct primes $p_{nm}>\max(3,n,m)$, $1\leq n,m\leq N$, so that
 \begin{equation*}
  H_2(M,\ZZ)= \ZZ^2 \oplus {\textstyle\bigoplus\limits_{n,m=1}^N }
 \left(\ZZ^{18n^2+2}_{p_{nm}}  \oplus \ZZ^{18m^2+2}_{p_{nm}^2} \oplus \ZZ^{20}_{p_{nm}^3}\right) .
 \end{equation*}
\end{theorem}

The strategy to prove Theorem \ref{thm:main} is as follows.
With constructions of symplectic geometry \cite{Gompf, 
CFM, McDuff}, we manage to obtain a symplectic orbifold with $b_2(X)=3$ for which 
$b_2^+(X)=3$, and for which there are $3$ disjoint symplectic surfaces
$D_1,D_2,D_3$. This is impossible for K\"ahler manifolds since in that case it always holds $b_2^+=1$. To translate the information of $b_2^+>1$ to the Seifert bundle $M\to X$, we take tuples of
$3$ disjoint symplectic surfaces $\varepsilon_{nm}=(D_1^n, D_2^m, D_3)$, parametrized
by natural numbers $n,m\geq 1$, where $[D_1^n]=n[D_1]$ and $[D_2^m]=m[D_2]$ (note that we could have used also multiples of $D_3$, but using two of the surfaces is enough). Taking the multiples
to obtain new surfaces is allowed since they have {\em positive} self-intersection. Here the fact that $b_2^+(X)>1$ plays a crucial role.

For each tuple $\varepsilon_{nm}=(D_1^n, D_2^m, D_3)$, we put
isotropy coefficients $p_{nm},p_{nm}^2,p_{nm}^3$, which are distinct and not coprime, in accordance with the fact that these are $3$ disjoint symplectic surfaces generating $H_2(X, \QQ)$.
The genera are computed easily from the genera of $D_1,D_2,D_3$ by the adjunction formulas. Therefore the manifold of Theorem \ref{thm:main} is obtained as a Seifert bundle once we fix $N>0$ and take the 
isotropy locus to be all curves in $\varepsilon_{nm}$ for $1\leq n,m\leq N$, with the given isotropy coefficients.

The proof that the $5$-manifold $M$ in Theorem \ref{thm:main} does not admit a Sasakian structure requires to check that
there is no singular complex surface $Y$ with cyclic singularities with $b_2(Y)=3$, and a large number of tuples $\varepsilon'_{l}
=(D_1^l,D_2^l,D_3^l)$ consisting of $3$ disjoint complex curves, which
will form the isotropy locus of the Seifert bundle $M\to Y$.
First one needs to bound the number of singular points (universally, i.e.\ independently of $Y$). This serves to bound geometric quantities, like the Euler characteristic, $K^2$,
or the self-intersection of negative curves. For orbifolds, the intersection and self-intersection numbers and $K^2$ can
be rational (instead of just integers), so it is useful to bound the denominators (independently of $Y$). 
As the number of singular points is bounded, one has that 
most of the tuples of disjoint complex curves avoid the singular points, and those satisfy the adjunction equality and have self-intersection which is an integer. 
For complex curves in a Kähler surface there are many more restrictions, and it would be a rare phenomenon that the tuples of curves were multiples of $3$ fixed curves (that is what we did for the symplectic orbifold $X$) because $b^+_2(Y)=1$. Noting that $\varepsilon'_{l}
=(D_1^l,D_2^l,D_3^l)$ gives an orthogonal basis of $H_2(Y,\QQ)$,
we have that there are many different orthogonal bases (non-proportional to each other), and writing $K^2$ with respect to each basis, we get a collection of
diophantine equalities and inequalities, which become impossible by 
choosing $N$ large enough.

\medskip

The purpose of the present contribution is to settle some ground work
on several questions on symplectic $4$-manifolds that is very useful
to make the constructions of cyclic symplectic $4$-orbifolds that may produce K-contact Smale-Barden manifolds
in the way they have been constructed in \cite{CMRV,Mu-jems, MRT}.

We need to produce large collections of symplectic surfaces in a symplectic $4$-manifold,
and they have to intersect transversally and positively. We will show how these collection of curves can be arranged when many of them intersect at a single point.
This can be used to perform a symplectic blow-up and separate them increasing $b_2$ only in one unit.

Next we shall prove that a collection of curves can be
made to intersect nicely (that is, like transverse intersection of complex curves), and only with 
pairwise intersections. Such collections are needed to be 
the ramification locus of a Seifert bundle over the $4$-manifold. 

In the next section, we provide K\"ahler models for neigbhourhoods of the patterns of symplectic surfaces constructed previously. 
This is very useful to apply techniques from complex geometry. For instance, when we have a 
chain of rational curves, this may be contracted to produce a cyclic singular point. This is needed for constructing 
symplectic cyclic $4$-orbifolds.

Finally, in section \ref{section:symplectic-divisors} we study two constructions of symplectic surfaces. The first one allows to construct
symplectic surfaces $C_n$ with $[C_n]=n\,[C]$, $n\geq 1$,
representing the multiples in homology of a single symplectic surface $C$ with positive self-intersection. This was used in the construction above of the basis $\e_{nm}$. 
The second construction allows to construct a symplectic surface representing a divisor $D=\sum a_iC_i$, in a neighbourhood of a collection of symplectic
surfaces $C_i$. This can be used to construct symplectic surfaces near a chain of rational curves that will be
blow-down to a cyclic singular point.

We hope that these constructions will be useful in symplectic geometry, and addressing the geography problem of
K-contact and Sasakian Smale-Barden manifolds. They certainly clarify some sloppy arguments and serve to complete the details appearing in previous works.

\medskip

\noindent {\bf Acknowledgements.} 
The first author was partially supported 
by Project MINECO (Spain)  PID2020-118452GB-I00.


\section{Model for many 
symplectic surfaces intersecting at a point}

The purpose of this and next sections is to give various constructions of symplectic surfaces in a symplectic $4$-manifold $(X,\omega)$, which have transversal intersections. These can be useful to produce isotropy locus of Seifert fibrations $M\to X$ over symplectic $4$-manifolds, which serve to construct K-contact $5$-manifolds $M$ with given homology $H_2(M,\ZZ)$. 

The constructions are local, that is, we do not assume $X$ to be compact. This implies in particular that they can be used for the case where $X$ is a cyclic singular symplectic $4$-manifold, by taking the smooth locus $X^o=X-P$, where $P\subset X$ is the collection of (finitely many) singular points. 
The constructions are useful in symplectic geometry {\em per se},
for the understanding of symplectic $4$-manifolds.

\medskip

Let $(X,\o)$ be a smooth symplectic $4$-manifold.
A collection of symplectic surfaces $S_i \subset X$ intersect transversely and positively if for $i \ne j$, at each $p \in S_i \cap S_j$ 
we have $T_p S_i \oplus T_p S_j=T_pX$ with the orientation of $T_pX$ given by $\cB_i \cup \cB_j$, with $\cB_i, \cB_j$ positive bases for $S_i, S_j$.
Here a symplectic surface $S \subset X$ 
is always oriented by taking the
symplectic form $\o_{S} =\o|_S$.

We want to understand what happens when we have several 
symplectic surfaces $S_i\subset X$, $0\leq i\leq l$, 
intersecting transversely and positively at a point. 
We aim for a model with complex equations.

\begin{definition} \label{def:complex-like}
 Let $S_i\subset X$, $0\leq i\leq l$, 
 be a collection of symplectic surfaces
 intersecting 
 at a point $p\in X$. We say that they have a complex-like
 intersection if there is a Darboux chart $(z,w)\in \CC^2$
 so that $S_i=\{w=a_i z\}$, where $a_i$ are distinct complex numbers.
\end{definition}

Recall that the standard symplectic structure of $\CC^2=\RR^4$ is given by $\o_0=\frac{\ii}{2}(dz \wedge d \bar z+dw \wedge d \bar w)$.
We start with the following linear algebra results for symplectic subspaces.

\begin{lemma} \label{lem:linear-algebra-1}
In $(\CC^2,\o_0)$, 
consider the planes 
$S_0=\{w=0\}$ and $S_1=\{w=az+b \bar z\}$, 
where $a,b\in\CC$. We have
\begin{enumerate}
    \item[(1)] $S_1$ is symplectic $\iff |a|^2-|b|^2+1 \ne 0$.
    \item[(2)] $S_0$, $S_1$ intersect transversely $\iff |a|^2-|b|^2 \ne 0$.
    \item[(3)] $S_1$ is symplectic and $S_0$, $S_1$ intersect transversely and positively 
    $\iff |a|^2-|b|^2 \in (-\infty,-1) \cup (0,\infty)$.
\end{enumerate}
\end{lemma}

\begin{proof}
A simple computation gives 
 $$
 \o_0|_{S_1}=(1+|a|^2-|b|^2) \ii\,  dz\wedge d \bar z \, ,
 $$
and this proves (1).

For (2), the condition that $S_0$, $S_1$ intersect transversely is the same as asking that 
$0 \ne \o_0^2(\cB_{0}, \cB_{1}) 
=\ii\, dw \wedge d \bar w(\cB_1)$,
with $\cB_0$ is the standard basis of $S_0$, and $\cB_1$ any basis of $S_1$. This translates to 
\[
0 \ne \ii\, dw \wedge d \bar w|_{S_1} =(|a|^2-|b|^2) \ii\, dz \wedge d \bar z|_{S_1}\, ,
\]
i.e. $|a|^2-|b|^2 \ne 0$. 

Finally for (3), $S_0$ and $S_1$ intersect 
transversely and positively if and only if 
$0 < \o_0^2(\cB^+_{0}, \cB^+_{1}) =\ii\, dw \wedge d \bar w(\cB^+_1)$,
where $\cB_0^+$ the standard basis of $S_0$, and $\cB_1^+$ a positive
basis for $S_1$, that is a basis so that  
\[
0<\o_0|_{S_1}(\cB_1^+)= (1+|a|^2-|b|^2) \ii \,dz \wedge 
d \bar z (\cB_1^+) = \frac{1+|a|^2-|b|^2}{|a|^2-|b|^2} \ii \,
dw \wedge d \bar w (\cB_1^+),
\]
and this translates to the condition  $\frac{1+|a|^2-|b|^2}{|a|^2-|b|^2}>0$.
\end{proof}

It is worth mentioning that the orientation of $S_1$ as a symplectic surface may not coincide with the orientation as a graph. It is interesting to see when this happens:
\begin{lemma} \label{lem:graphs}
Let $S_0=\{w=0\}$ and $S_1=\{w=az+b\bar z\}$ be planes of $\CC^2$. Assume both surfaces are endowed with the orientation as graphs on $z$. Then, $S_1$ and $S_0$ intersect transversely and positively if and only if $|a|^2-|b|^2>0$.
\end{lemma}
\begin{proof}
By hypothesis we know that $\ii dz \wedge d \bar z(\mathcal{B}_{S_1}^+)>0$. On the other hand, $\o_0^2(\mathcal{B}_{S_0}^+,\mathcal{B}_{S_1}^+)=\ii dw \wedge d \bar w(\mathcal{B}_{S_1}^+)=\ii (|a|^2-|b|^2) d z \wedge d \bar z(\mathcal{B}^+_{S_1})$, and this gives $|a|^2-|b|^2>0$.
\end{proof}

\begin{corollary} \label{cor:linear-algebra0}
Let $S_1=\{w=a_1z+b_1\bar z\}$, $S_2=\{w=a_2 z + b_2 \bar z\}$ planes of $\CC^2$, oriented as graphs. Then $S_1$ and $S_2$ intersect transversely and positively $\iff |a_1-a_2|^2-|b_1-b_2|^2 >0$.
\end{corollary}

\begin{proof}
Consider the map $\p (z,w)=(z,w-a_1z-b_1 \bar z)=(z',w')$, which preserves the orientation of $\CC^2$ and transforms $S_1$ to $\p(S_1)=\{w'=0\}$ and $S_2$ to $\p(S_2)=\{w'=(a_2-a_1)z'+(b_2-b_1) \bar z'\}$. Hence $S_1$ and $S_2$ intersect transversely and positively if and only if $\p(S_1)$ and $\p(S_2)$ do. On the other hand, $\p$ preserves the coordinate $z$, so the induced orientation of $\p(S_1)$ and $\p(S_2)$ are the orientation as graphs on $z'$. Now Lemma \ref{lem:graphs} gives the result. 
\end{proof}

If the surfaces $S_0$ and $S_1$ are given as graphs on different coordinates (say $S_0$ on $z$ and $S_1$ on $w$), we have a simpler result. In fact, in this case, the orientation as a symplectic surface does coincide with the orientation as a graph:

\begin{lemma} \label{lem:linear-algebra-2}
In $(\CC^2,\o_0)$, 
consider the planes $S_0=\{w=0\}$ and $S_1=\{z=\a w+\b \bar w\}$, where $\a,\b\in\CC$. We have
\begin{enumerate}
    \item[(1)] $S_1$ is symplectic $\iff |\a|^2-|\b|^2+1 \ne 0$.
    \item[(2)] $S_0$, $S_1$ intersect transversely for any $\a,\b \in \CC$.
    \item[(3)] $S_1$ is symplectic and $S_0$, $S_1$ intersect transversely and positively 
    $ \iff |\a|^2-|\b|^2+1>0$
    \item[(4)] In any of the hypothesis from (3), it follows that $S_1$ is oriented as a graph.
\end{enumerate}
\end{lemma}

\begin{proof}
In this case, (1) follows from
$$
 \o_0|_{S_1}=(1+|\a|^2-|\b|^2) \ii\, dw \wedge d\bar w\, ,
 $$

To prove (2),  for $\cB_0$ the standard basis of $S_0$ and 
$\cB_1$ any basis of $S_1$, we have
$\o_0^2(\cB_0,\cB_1)= \ii\, dw \wedge d \bar w(\cB_1) =
\ii\, dw \wedge d \bar w|_{S_1}$, which is always non-zero since $w$ is the coordinate for $S_1$.

To see (3), note that $S_0$ and $S_1$ intersect transversely and positively if and only if
$0 < \o_0^2(\cB^+_{0}, \cB^+_{1}) =\ii \, dw \wedge d \bar w(\cB^+_1)$,
with $\cB_1^+$ a positive basis for $S_1$, which means that
\[
0<\o_0|_{S_1}(\cB_1^+)= (1+|\a|^2-|\b|^2) \ii \,dw \wedge d \bar w (\cB_1^+),
\]
and this gives $1+|\a|^2-|\b|^2>0$.

Finally, (4) is clear noting that $S_1$ is oriented as a graph if and only if $0<\ii \, dw \wedge d \bar w(\cB^+_1)$.
\end{proof}
\begin{remark}
Note that when the surface $S_1$ can be represented both as $w=az+b\bar{z}$, and $z=\a w+\b \bar{w}$, then we have
 $$
 \a= \frac{\bar{a}}{|a|^2-|b|^2}, \qquad 
 \b= \frac{-{b}}{|a|^2-|b|^2},
 $$
so that
$$
 |\a|^2-|\b|^2+1 = \frac{1+|a|^2-|b|^2}{|a|^2-|b|^2}\, .
 $$
Therefore the statements of Lemmas \ref{lem:linear-algebra-1} and \ref{lem:linear-algebra-2} are equivalent.
\end{remark}

When three planes intersect at the origin, we get a stronger condition.

\begin{proposition} \label{prop:linearalgebra1}
Let $S_0=\{z=0\}$, $S_1=\{w=a_1z+b_1\bar z\}$, $S_2=\{w=a_2 z + b_2 \bar z\}$ be planes of $(\CC^2,\o_0)$.
Then all of them are symplectic and intersect transversely and positively (pairwise) if and only if 
\begin{itemize}
    \item $|a_1|^2-|b_1|^2+1>0$, $|a_2|^2-|b_2|^2+1>0$,
    \item $|a_1-a_2|^2-|b_1-b_2|^2>0$.
\end{itemize}
\end{proposition}

\begin{proof}
By Lemma \ref{lem:linear-algebra-2} applied to $S_0$ and $S_j$, $j=1,2$, we have the conditions $|a_j|^2-|b_j|^2+1>0$
for $S_j$ to be symplectic and for $S_0,S_j$ to intersect transversely and positively. This implies also that $S_1, S_2$ are oriented as graphs. Hence by Corollary \ref{cor:linear-algebra0} we deduce that $S_1, S_2$ intersect transversely and positively if and only if $|a_1-a_2|^2-|b_1-b_2|^2>0$.
\end{proof}

In the following theorem we see that for any finite family of symplectic surfaces intersecting transversely
and positively at a point $p$, up to a $C^0$-small deformation, we arrange  that the surfaces are complex in a small chart around $p$, that is, they intersect complex-like according to Definition \ref{def:complex-like}.

\begin{theorem} [Complex model for a multiple intersection] \label{thm:complex-model-many-curves}
Let $(X,\o)$ be a symplectic $4$-manifold, and suppose that $S_0, S_1, \dots, S_l \subset X$
are symplectic surfaces intersecting transversely and positively all of them (pairwise) at the same point $p \in X$.
Then we can deform the surfaces $S_j$ to $\tilde{S}_{j}$ for $j \ge 1$, such that:
\begin{enumerate}
\item The perturbed surfaces $\tilde{S}_{j}$ are symplectic for $1 \le j \le l$,

\item the deformation of the surfaces $S_j$ is $C^0$-small and only occurs in a small chart around $p$,
and the surfaces $S_0$ and $\tilde{S}_j$, $1\le j\le l$, only intersect at $p$ (in the given chart).

\item Near the point $p$ there are Darboux coordinates $(z,w)$ 
for which $S_0=\{z=0\}$ and $\tilde{S}_{j}=\{w=a_j z\}$, for $1 \le j \le l$, where $a_1,\ldots, a_l\in \CC^*$ are distinct complex numbers.
\end{enumerate}
\end{theorem}

\begin{proof}
We can take Darboux coordinates $(z,w)$ adapted to $S_0$ so that $p=(0,0)$, $S_0=\{z=0\}$, $S_j=\{w=a_jz+b_j \bar z + R_j(z)\}$ for some $a_j, b_j \in \CC$, and 
 $$
  |R_j(z)|\le C|z|^2, \quad |\partial_z R_j|
  \le C |z|, \quad |\partial_{\bar z} R_j| \le C |z|,
  $$
for some constant $C>0$. These coordinates are defined in some ball $B^4(r_0)$. 
By Proposition \ref{prop:linearalgebra1}  
applied to the tangent planes at $p$, we have:
\begin{itemize}
    \item $|a_j|^2-|b_j|^2+1>0$, for all $j$,
    \item $|a_j-a_k|^2-|b_j-b_k|^2 >0$, for all $j\neq k$.
\end{itemize}
Denote $M=\max_j \{|a_j|, |b_j| \}$, and 
\[
\e_0=\min_{j \ne k} \{1, |a_j|^2-|b_j|^2 + 1,
 |a_j-a_k|-|b_j-b_k|\}>0 \, .
\]
Let
$\rho: \RR \to [0,1]$ be an increasing function vanishing for $t \le 1$, $|\rho'| \le 2$, and $\rho \equiv 1$ on $\{t \ge 2\}$.

\medskip

\noindent \textbf{Step one: linearize the $S_j$}.
Denote 
\[
S'_j=\{w=a_jz+b_j \bar z+\rho(\tfrac{|z|}{\l}) R_j(z)\},
\]
with $\l>0$ to be chosen later. An initial requirement for 
$\l$ is that the set 
$$
\{(z,w) \mid  |z| \le 2\l, w=a_jz+b_j \bar z + \rho(\tfrac{|z|}{\l}) R_j(z)\} \subset B^4(r_0).
$$
In this set $|z|\le 2\l, |w| \le 2M\l + 4C \l^2$, so the condition is $4\l^2+ (2M\l + 4 C \l^2)^2 < r_0$.

Clearly $S'_j \cap S_0=\{(0,0)\}$. 
Let us see that $S'_j \cap S'_k=\{(0,0)\}$. For $|z| \le \l$ and $|z| \ge 2 \l$ this is clear (in the first case, they are linear subspaces; in the second case they are the original surfaces $S_j,S_k$ that do not intersect by hypothesis).
At any point of intersection with $\l \le |z| \le 2\l$ we would have
\begin{align*}
0 & =|(a_j-a_k)z+(b_j-b_k) \bar z + \rho(\tfrac{|z|}{\l}) (R_j(z)-R_k(z))| \\
& \ge (|a_j-a_k|-|b_j-b_k|)|z| - \rho(\tfrac{|z|}{\l}) |R_j(z)-R_k(z)| \\
& \ge \e_0 |z| - 2C |z|^2\ge \e_0 \l -8C \l^2=\l(\e_0-8C\l),
\end{align*}
so it suffices to take $\l<\frac{\e_0}{8C}$ to see that there are no such points.

Let us see that $S'_j$ is symplectic. Again this is clear for $|z| \le \l$ and for $|z| \ge 2\l$. Take a point of $S'_j$ with $\l \le |z| \le 2 \l$, and compute
\begin{align*}
\partial_z w & =a_j+\tfrac{\bar z}{2 \l |z|} \rho' R_j+\rho \, \partial_z R_j=a_j + E_1(z), \\
\partial_{\bar z} w & =b_j+\tfrac{z}{2 \l |z|} \rho' R_j+\rho \, \partial_{\bar z} R_j=b_j+E_2(z),
\end{align*}
with $|E_i(z)| \le 
\frac{2C|z|^2}{2\lambda}+C|z| \le 6C\l$. 
Given this, we have
\begin{align*}
& 1+|\partial_z w|^2-|\partial_{\bar z} w|^2
 = 1+ |a_j+E_1|^2- |b_j+E_2|^2  \\
 & \ge   1+|a_j|^2-|E_1|^2 - 2 |a_j||E_1|- |b_j|^2-|E_2|^2-2|b_j||E_2|\\
  & \ge   \e_0 - 72 C^2\lambda^2- 24 MC\lambda 
  \ge \e_0 - C(9+24M)\lambda,
\end{align*}
using that $M=\max_j\{|a_j|,|b_j|\}$, and that $8C\lambda<
\e_0\le 1$, as proved earlier. 
So it suffices to take $\l<\frac{\e_0}{C(9+24M)}$, to
obtain that $1+|\partial_z w|^2-|\partial_{\bar z} w|^2
 >0$, and hence $S_j'$ is symplectic.

\medskip

\noindent \textbf{Step two: remove the anti-holomorphic part}.
By the previous step, we have Darboux coordinates (in a smaller chart) so that $S_0=\{z=0\}$, $S'_j=\{w=a_jz+b_j \bar z \}$. Note that the $a_j$ are all distinct by the condition $|a_j-a_k|^2-|b_j-b_k|^2>0$.
Define 
\[
\tilde S_j=\{w=a_j z + b_j \bar z \, \rho(\l \log |z|+ \mu)\},
\]
where $\l>0$ is small to be chosen later, and $\mu=\mu(\l) \ge 2$ is chosen after $\l$ so that the set where 
$\l \log |z|+ \mu \le 2$ is contained in the Darboux chart, i.e. the set 
\[
\{(z,w) \mid |z| \le  \exp\left(
{\frac{2-\mu}{\l}}\right), |w| \le 2M
\exp\left({\frac{2-\mu}{\l}}\right) \}
\] 
is inside the chart, with $M=\max_j\{ |a_j|,|b_j|\}$.

Let us see first that $\tilde S_j \cap \tilde S_k=\{(0,0)\}$. We already know this  for $z \le \exp(\frac{1-\mu}{\l})$ since $\tilde S_j=\{w=a_jz\}$, $\tilde S_k=\{w=a_k z\}$ in such locus. 
The same happens for $z \ge \exp(\frac{2-\mu}{\l})$ since $\tilde S_j=S'_j$, $\tilde S_k=S'_k$ do not intersect there. 
Consider now $\exp(\frac{1-\mu}{\l}) \le |z| \le \exp(\frac{2-\mu}{\l})$, and suppose that
$\tilde S_j$, $\tilde S_k$ intersect, so
\[
0= |(a_j -a_k)z + (b_j -b_k)\rho\, \bar z| 
\ge (|a_j -a_k|-|b_j -b_k|) |z| \ge \e_0  \exp(\tfrac{1-\mu}{\l})>0,
\]
giving a contradiction. 

Finally we see that $\tilde S_j$ is symplectic. Again, we only need to see it for $|z|$ in the annulus
$\exp(\frac{1-\mu}{\l}) \le |z| \le \exp(\frac{2-\mu}{\l})$.
We have that $\partial_z w=a_j+b_j \bar z  \rho' \frac{\l }{|z|} \frac{\bar z}{2 |z|}$, and 
$\partial_{\bar z} w=b_j \rho + b_j \bar z  \rho' \frac{\l }{|z|} \frac{ z}{2 |z|}$, so we have
\begin{align*}
 1&+|\partial_z w|^2 - |\partial_{\bar z} w|^2 \geq \\
& \ge 1+|a_j|^2- |b_j|^2 (\rho')^2 \tfrac{\l^2}{4}- |a_j||b_j|\rho' \l -|b_j|^2 \rho^2-
  |b_j|^2 (\rho')^2 \tfrac{\l^2}{4}- |b_j|^2 \rho \rho'\l
  \\
& \ge 1+|a_j|^2-|b_j|^2 - 2 M^2\lambda^2 -4M^2\lambda
 \\
& \ge \e_0 - 2 M^2\lambda^2 -4M^2\lambda\, .
\end{align*}
Taking $\l<1$, we need to take $\l < \frac{\e_0}{6 M^2}$. With both conditions satisfied, we have that 
after this perturbation,
we end up with a smaller Darboux chart so that $S_0=\{z=0\}$, $\tilde S_j=\{w=a_j z\}$, as desired.
\end{proof}

\begin{remark} \label{rem:change-a_j}
With a slight modification of the proof, we can moreover see that the coefficients $a_j$ of the surfaces $S_j$ can be slightly changed, so as to avoid $0$ for instance.
Indeed, after the first step we have $S_j=\{w=a_jz+b_j\bar z\}$. Pick for instance $j=2$. We deform $S_2$ to 
\[
\hat S_2=\{w=\big(a_2 + \e \big(1-\rho(\tfrac{|z|}{\a_0})\big)\big)z + b_2 \bar z\} \, ,
\]
where $0<\e<1$ is small to be chosen later, and $\a_0$ is chosen so that the set where $|z| \le 2 \a_0$ and $w=
\big(a_2 + \e \big(1-\rho(\frac{|z|}{\a_0})\big)\big)z 
+ b_2 \bar z$ is contained in $B^4(r_0)$. 
This is easily done as $|w| \le (2M+1)|z|$.

Note that $\hat S_2=S_2$ for $|z| \ge 2 \a_0$, and $\hat S_2=\{w=(a_2 + \e )z + b_2 \bar z\}$ for $|z| \le \a_0$. 
The surface $\hat S_2$ intersects the other 
$S_j$ only at $(0,0)$ if we take $\e>0$ small, as
\begin{align*}
|(a_2 + &\e(1-\rho)-a_j)z+(b_2-b_j) \bar z| \ge \\
& \ge (|a_2-a_j|-|b_2-b_j|-\e)|z| \ge (\e_0-\e)|z|\, ,
\end{align*}
so it suffices to take $\e \le \frac12 \e_0$.
On the other hand, as $\rho(\frac{|z|}{\a_0})$ and its derivative are bounded for $\alpha_0$ fixed, 
the surface $\hat S_2$ is arbitrarily $C^1$-close to $S_2$ as $\e \to 0^+$. 
We choose $\e>0$ small so that $\hat S_2$ is symplectic, $a_2+\e \ne a_j$ for all $j$, and $\hat S_2$ still intersects the other surfaces transversely and positively at the point $p$. After this modification, the proof of Theorem \ref{thm:complex-model-many-curves} goes on with step two.
\end{remark}

In Theorem \ref{thm:complex-model-many-curves}, the surface $S_0$ plays a special role since it has vertical slope, 
but we can state a more symmetrical role for all surfaces as in the following.

\begin{corollary} \label{cor:complex-model-many-curves}
Let $(X,\o)$ be a symplectic $4$-manifold, and suppose
that $S_0, S_1, \dots, S_l \subset X$ 
are symplectic surfaces as in Theorem 
\ref{thm:complex-model-many-curves}.
We can deform the surfaces $S_j$ so that near the point $p$ there are Darboux coordinates $(z,w)$ 
in which $\tilde S_0=\{w=\a_0 z\}$ and $\tilde S_{j}=\{w=\a_j z\}$ for $1 \le j \le l$, for some $\a_0,\a_j \in \CC$, distinct complex numbers.
\end{corollary}

\begin{proof}
We apply Theorem \ref{thm:complex-model-many-curves} with the coordinates $z,w$ reversed, and we get Darboux coordinates so that the perturbed surfaces are $S_0=\{w=0\}$ and $S_j=\{z=a_j w\}$. 
We can moreover assume that $a_j \ne 0$ for all $j$, because if one were zero we can apply the technique from remark \ref{rem:change-a_j} above to 
slightly change it. Then $w=\frac{1}{a_j} z$. To finish, we can apply a unitary linear map that transforms $\{w=0\}$ into $\{w=\a_0 z\}$, 
e.g.\ $A=\frac{1}{1+|\a_0|^2} 
\begin{pmatrix}
1 & -\bar \a_0 \\
\a_0 & 1
\end{pmatrix}
$ does the job. The rest of the surfaces are transformed by this map to other surfaces of type $w=\a_j z$.
\end{proof}

\begin{remark}
Note that in the Corollary \ref{cor:complex-model-many-curves} we cannot choose all the coefficients $\a_0,\a_1, \dots, \a_l$ arbitrarily. 
We can choose only one of them, say $\a_0$. The rest $\a_1, \dots , \a_l$ are not completely rigid, 
as they can be slightly changed, however we cannot control how far they can be changed.
\end{remark}

\section{Arranging that all intersections are nice} 

Recall the following definition from \cite{MRT}.

\begin{definition} \label{def:nice}
We say that a family $S_i$ 
of symplectic surfaces in $X$ intersect nicely if they intersect transversely and positively, no three of them intersect at the same point, and in addition around each point of intersection of $S_i \cap S_j$ there exists a Darboux chart $(z,w)\in \CC^2$ so that $S_i=\{w=0\}$, $S_j=\{z=0\}$.
\end{definition}

The model in Definition \ref{def:nice} means that $S_i,S_j$ intersect orthogonally. 
When we have two surfaces, 
after a $C^0$-small deformation, we can obtain that all positive and transverse intersections of two symplectic surfaces are nice, as proved in the following result.

\begin{lemma}[{\cite[Lemma 6]{MRT}}] \label{lem:orthogonal}
Let $S_0, S_1 \subset X$ be symplectic surfaces intersecting transversely and positively at a point. Then there is $C^0$-small perturbation of the surfaces (actually of $S_1$)
in a small neighborhood of the intersection point so that the perturbed surfaces become $S_0=\{z=0\}$, 
$S_1'=\{w=0\}$.
\end{lemma}

In order to get nice intersections for a family of symplectic surfaces, we need to get rid of multiple intersections. We do this in the following proposition. 

\begin{proposition}[From transverse and positive to 
nice intersections] 
\label{prop:transverse-positive-implies-nice}
If $S_j$ is a finite family of symplectic surfaces of $X$ intersecting transversely and positively (pairwise), then after a $C^0$-small perturbation of the surfaces near the intersection points, the modified surfaces $\tilde S_j \subset X$ are symplectic, and they intersect nicely.
\end{proposition}

\begin{proof}
We need to see that we can perturb the surfaces $S_j$ so that we avoid multiple intersections. After this is done we are left with a family of symplectic surfaces $S_j$ so that at every intersection point there are two surfaces intersecting, so Lemma \ref{lem:orthogonal} applies to obtain orthogonal (hence nice) intersections.

Lat $p$ be a point of intersection, and let us denote $S_0, S_1, \dots, S_l$ the surfaces that pass through $p$, where $l\ge 2$. By Theorem \ref{thm:complex-model-many-curves}, after a suitable perturbation of $S_1, \dots , S_l$ we get coordinates $(z,w)$ around $p=(0,0)$ so that the $S_j$ have local equations of the form
\[
S_i=\{w=\alpha_i z\} , \quad i=0,\dots , l,
\]
where $\alpha_0=0$, $\alpha_i \in \CC^*$ are all distinct.
This shows that the surfaces $S_j$ are the graphs of the functions $s_j(z)=\alpha_j z$. Let $r_0>0$ be such that the Darboux chart is the ball $B^4(r_0)$, and $\lambda_0>0$ so that the graphs of the surfaces for $|z|\le 2\lambda_0$ are
inside $B^4(r_0)$.

Now we shall modify $S_l$ so that it no longer passes through $(0,0)$ and it intersects the other surfaces transversely and positively at distinct points. Repeating this for $S_{l-1}, \ldots, S_2$ inductively, we end up with all surfaces intersecting only pairwise and transversely and positively.

For small $\e>0$ to be chosen later, define $\tilde s_l(z)= \alpha_l z+ \e \,\rho({\frac{|z|}{\lambda_0}})$, with $\rho(t)$ a bump function which equals $1$ for $t \le 1$, 
equals $0$ for $t \ge 2$, and it has 
$|\rho'(t)| \le 2$. This perturbation agrees with $s_l$ outside of the chart.
We have that 
\[
\tilde s_l(z)=s_i(z) \iff \alpha_l z+ \e \rho(
{|z|}/{\lambda_0})=\alpha_i z\, .
\]
If $|z| \ge \lambda_0$ this equality cannot happen if we choose $\e>0$ small, because
\[
|\alpha_l z+ \e \rho({{|z|}/{\lambda_0}})-\alpha_i z| 
\ge |\alpha_l-\alpha_i||z|-\e \rho(
{|z|}/{\lambda_0}) 
\ge |\alpha_l-\alpha_i| \lambda_0- \e>0 ,
\]
for a choice of $\e < 
\min_{i\ne l} \{|\alpha_l-\alpha_i|\}/\lambda_0$.
Hence the equality can only occur for $|z| \le \lambda_0$, where $\rho=1$.
The intersection point $S_i\cap \tilde S_l$ is 
 $$
 z_{il}=\frac{\e}{\alpha_i-\alpha_l} \,.
 $$ 
All of these points are distinct. 
Moreover, the intersection points are transverse and positive since the graphs are complex linear subspaces
$w=\alpha_i z$ and $w=\alpha_lz+\e$.

After this modification, we have that all intersections of $\tilde S_l$ with the rest of surfaces $S_i$ are transverse and positive, and occur at distinct points $p_{il}=(z_{il},\a_i z_{il})$. However, at $p=(0,0)$ there is a multiple intersection of $S_0,\ldots, S_{l-1}$. We can take a new 
(smaller) Darboux coordinate ball $B^4(r_1)$ around $p$
that satisfies that it does not intersect the previous surface $\tilde S_l$. We
repeat the process inside this ball with $S_{l-1}$, so
that we get a perturbed surface $\tilde S_{l-1}$ intersecting $S_0,\ldots, S_{l-2}$ transversely and positively, at distinct points. The intersection $\tilde S_{l-1}\cap \tilde S_l$ is the same as $S_{l-1}\cap \tilde S_l$ that is already transverse and positive.

We iterate the process until we end up with $S_0$ and $S_1$ through $p$ alone, and we are done.
\end{proof}

\begin{remark}
  Proposition \ref{prop:transverse-positive-implies-nice} can be proved alternatively by applying standard transversality arguments. 
  After a small \emph{generic} $C^1$-perturbation of 
  $S_0,\ldots, S_l$ we can obtain surfaces $\tilde S_0,\ldots, \tilde S_l$ so that 
  all intersections are transverse and at distinct points. 
  As the perturbation is $C^1$-small, the perturbed surfaces $\tilde S_i$ are still symplectic. The positivity of the 
  intersections is preserved under $C^1$-small perturbations; still one has to keep some control so that no new intersections appear.
  
Our argument is however more explicit. And moreover, the 
structural result in Corollary \ref{cor:complex-model-many-curves} can be useful 
in other situations, for instance when one wants to separate many symplectic surfaces by a single symplectic
blow-up. 
\end{remark}

Once that we have nice intersections, we can apply the 
process of symplectic resolution of nice intersections. This
appears in \cite[section 5.1]{MRT}. As this reference does not
include a full proof, we provide it here for future reference.

\begin{lemma}[Resolution of nice intersections]
\label{lem:resolution}
Let $(X,\o)$ be a symplectic $4$-manifold, and 
let $S_1, S_2\subset X$ be two symplectic surfaces intersecting nicely at a point $p$. Then there is
a symplectic surface $S\subset X$ diffeomorphic to
the 
connected sum $S_1\# S_2$. The gluing of $S_1$ and $S_2$ is performed
in a small neighbourhood around $p$, so $[S]=[S_1]+[S_2]$ in $H_2(X,\mathbb{Z})$.
\end{lemma}

\begin{proof}
Consider a Darboux chart $(z,w)$ such that $S_1=\{w=0\}$
and $S_2=\{z=0\}$. Take a smooth bump function
$\rho(t)$ such that $\rho(t)\equiv 1$ for $t\in [0,1]$,
$\rho(t)\equiv 0$ for 
$t\in [2,\infty)$ and $|\rho'|\leq 2$.
Take the surface 
$$
 S=\left\{(z,w) \, | \, 
 zw=\e \, \rho
 \left(\frac{\sqrt{|z^2|+|w^2|}}{\lambda_0}\right)\right\}
 $$
 where the constant $\lambda_0>0$ is taken so that the ball $B^4(2 \lambda_0)$ lies inside the Darboux chart, and $\varepsilon>0$ is small to be chosen later.
For $(z,w) \notin B^4(2\lambda_0)$, i.e. for $|z|^2+|w|^2>4 \lambda_0^2$, we have $S=\{zw=0\}=S_1 \cup S_2$. For $(z,w)\in B^4(\lambda_0)$, $S=\{zw=\e \}$
which topologically is an annulus that has substituted
the two discs $zw=0$ intersecting at the origin. 

To see that $S$ is smooth, we only need to check it for $(z,w)\in B^4(2\lambda_0)-B^4(\lambda_0)$.
Write $f_\varepsilon(z,w)=zw-\e \,\rho(r/\lambda_0)$,
with $r=\sqrt{|z|^2+|w|^2}$, and let $f(z,w)=zw$. Then it is easy to see that the $C^1$-norm $\lvert f_\varepsilon-f \rvert \le C \e$ for some constant $C$. Hence for $\e>0$ small, the differential of $f_\varepsilon$ has rank 2 in $B^4(2\lambda_0)-B^4(\lambda_0)$, because the differential of $f$ does.


Finally, $S$ is symplectic. This is clear outside
$B^4(2\lambda_0)$, and in $B^4(\lambda_0)$. In the annulus
$B^4(2\lambda_0)-B^4(\lambda_0)$, $f(z,w)=zw$ defines
a symplectic surface, and as $f_\varepsilon(z,w)$ is a small $C^1$-deformation, then
$S=\{f_\varepsilon(z,w)=0\}$ is a symplectic surface.
\end{proof}

\section{K\"ahler model for tubular neighborhoods} \label{section:kahler-model}

Now we are going to revisit the construction of symplectic plumbings from \cite{CMRV},  
and we will see that the resulting manifold can also be given the structure of a complex K\"ahler manifold.

We start with a compact Riemann surface, that is a surface 
$(S,\o_S,J_S)$ with a K\"ahler structure. This means that $S$ is symplectic and 
it is also endowed with a compatible almost complex 
structure $J_S$. This produces a Riemannian metric $g_S(u,v)=\omega_S(u,J_S(v))$ and a hermitian metric 
$h_S=g_S+ \ii\, \omega_S$.

Now let $\pi:L \to S$ be a holomorphic line bundle over $S$. 
This means that we have a covering $\{B_\a\}$ of $S$ such that
$L|_{B_\a} =\pi^{-1}(B_\alpha) \cong B_\a \times \CC$, and the changes of trivializations 
$g_{\a\b}:B_\a\cap B_\b \to \CC^*$ are holomorphic functions. There is a well-defined 
complex structure $J_L$ on the total space $L$. We also put a hermitian metric $h$ on $L$ and denote $B_r(S) \subset L$ 
the disc bundle of radius $r$ with respect to $h$.

\begin{lemma}[K\"ahler model for a neighbourhood of a surface] \label{lem:kahler-line-bundle}
For $r>0$ small enough, the disc bundle $B_r(S)\subset L$ admits a symplectic form $\o_L$ 
which is compatible with the complex structure $J_L$, defining thus a K\"ahler structure $(B_r(S),\o_L, J_L)$ so that the inclusion $S \to B_r(S)$ is K\"ahler.

Consider a finite collection of points $p_i \in S$. 
Suppose that $(\o_S,J_S)$ is the flat K\"ahler structure on small neighbourhoods $V_i\subset S$ of $p_i$
for all $i$.
Then we can choose a hermitian metric $h$ so that 
the trivializations $L|_{V_i}=
\pi^{-1}(V_i) \cong V_i \x \CC$ give
diffeomorphisms $B_r(S) \cap \pi^{-1}(V_i)  \cong V_i\x B_r(0)$ such that 
$(\o_L,J_L)$ is the product flat K\"ahler structure on it.
\end{lemma}

\begin{proof}
Let us call $\mathbf{r}: L \to [0,\infty)$ the radial function 
on the fibers defined as $\mathbf{r}(v)
=|v|=\sqrt{h(v,v)}$. We 
define the $2$-form on $L$ by the formula
\begin{align} \label{eq:omega_L}
 \o_L= \pi^* \o_S - \tfrac{1}{4} d J_L d (\mathbf{r}^2) \, .
\end{align}
A priori, this is well-defined and smooth 
for $\mathbf{r}\neq 0$. 
 In a holomorphic trivialization $L|_{B_\a} \cong
 B_\a \times \CC$ with coordinates $(z,w)$, the metric is $h=h(z) \, dw \cdot d \bar w$, 
 so $\mathbf{r}^2
 = h(z) |w|^2$. Using that $J_L(dz)=\ii\,dz$,
 $J_L(dw)=\ii\,dw$, 
 $J_L(d\bar z)=-\ii\,d\bar z$ and
 $J_L(d\bar w)=-\ii\,d\bar w$, we see that $\o_L$ has the local form
\begin{align*}
\o_L^{\a} =& \, \pi^* \o_S - \tfrac{1}{4} d J_L d (h(z) |w|^2) \\
=& \, \pi^* \o_S + \tfrac{\ii}{2} |w|^2 (\bd^2_{z \bar z} h) dz \wedge d \bar z 
+ \tfrac{\ii}{2} h(z) dw \wedge d \bar w  \\ 
& - \tfrac{\ii}{2} \bar w (\partial_{\bar z} h) d \bar z \wedge d w 
+ \tfrac{\ii}{2} w (\partial_z h) dz \wedge d \bar w \, .
\end{align*}
This shows that $\o_L$ is smooth, and it is 
a $(1,1)$-form, so $\o_L(J_L \cdot, J_L \cdot)=\o_L$. Also, when $w=0$ we have 
\[
\o_L^{\a}|_{(z,0)}
= \pi^* \o_S+\tfrac{\ii}{2} h(z) \, dw \wedge d \bar w \, ,
\]
hence $\o_L(\cdot, J_L \cdot)>0$ on $\{w=0\}$, so there is some neighborhood $B_r(S)$ of the zero section where $\o_L$ is $J_L$-tamed. The pair $(\o_L, J_L)$ defines a K\"ahler structure on $B_r(S)$, and the inclusion $(S,\o_S, J_S) \to (B_r(S), \o_L, J_L)$ respects the K\"ahler structure.

For the second part, assume that $(\o_S, J_S)$ is flat on some small neighborhood $V_i$ of the points $p_i$. 
Therefore $\o_S=\ii\, dz\wedge d\bar z$ on $V_i$. We can choose the Hermitian metric so that on the holomorphic trivialization $L|_{V_i}\cong V_i\x \CC$ we have $h(z)\equiv 1$, that is $h=dw\cdot d\bar w$. 
Then 
 $$
  \o_L=\pi^* \omega_S + \tfrac{\ii}{2}\, dw\wedge d\bar w
=\tfrac{\ii}{2}\, dz\wedge d\bar z+\tfrac{\ii}{2}\, dw\wedge d\bar w
$$
is the product symplectic form. Note that the chart  
$L|_{V_i}\cong V_i\x \CC$ is holomorphic, hence the complex
structure is also a product. This implies that also the
associated Riemannian metric is flat, so we have the flat K\"ahler product structure.
\end{proof}

\begin{remark}\label{rem:rem3}
Given a K\"ahler structure $(\o_S,J_S)$ on the Riemann surface $S$, we can always modify $\o_S$ to another $\o'_S$ 
so that $(\o'_S,J_S)$ is a flat K\"ahler structure near a finite set of points $p_i$. 
Indeed, take holomorphic coordinates $V_i$ around the points $p_i$, and substitute the symplectic form $\o_S$ by $\o'_S=\rho \o_S + (1-\rho) \o_0$ with $\rho$ a bump function vanishing on $V'_i \Subset V_i$ and $\rho \equiv 1$ 
outside $V_i$. Note that $\o'_S$ is closed, being a $2$-form on a surface, and it is $J_S$-positive in $V_i$ 
because both $\o_S$ and $\o_0$ are. In this way 
$(\o'_S,J_S)$ is a K\"ahler structure which becomes flat on
$V'_i$. We may assume that $(S,\o_S)$ and $(S, \o'_S)$ are symplectomorphic via a rescaling of $\o'_S$ so that the symplectic areas match.
 
Alternatively, we can modify $J_S$ to another $J_S'$ so that 
$(\o_S,J_S')$ is a flat  K\"ahler structure near $p_i$.  
For this, we take a Darboux chart around $p_i$, and consider the initial riemannian metric $g_S$ and the flat metric $g_0$, and substitute $g_S$ by $g_S'=\rho g_S+(1-\rho)g_0$. We note that $g_S'$ gives a conformal structure and hence an almost complex structure $J_S'$. For surfaces, all almost complex structures are integrable, so this gives a flat K\"ahler structure on a neighbourhood of $p_i$.
\end{remark}

Next we need to recall the following result from \cite[Proposition 13]{CMRV}, which we adapt to the current situation. It is the relative version of the classical symplectic tubular neighborhood.

\begin{proposition}[Symplectic tubular neighborhood] \label{prop:tubular1}
Let $(X,\o)$ and $(X',\o')$ be two symplectic $4$-manifolds with compact symplectic surfaces $S \subset X$ and $S' \subset X'$ which are symplectomorphic via 
$f \colon S \to S'$,
and suppose that $S^2=S'{}^2$. 
Then there are tubular neigbourhoods $\cV,\cV'$ 
of $S$ and $S'$, respectively, and a symplectomorphism
$g\colon \cV\to \cV'$ with $g|_S=f$. 

Moreover, if there are balls $W\subset \cV$, $W'\subset \cV'$ and a symplectomorphism $h \colon W\to W'$ such that 
$h|_{W\cap S}=f|_{W\cap S}\colon 
W\cap S \to W'\cap S'$, then we
can choose the symplectomorphism $g\colon \cV\to \cV'$ such that
$g|_{\hat{W}}=h$ on a slightly smaller ball 
$\hat{W}\Subset W$. \hfill $\Box$
\end{proposition}

The following result proves that tubular neighborhoods of nice configurations of symplectic surfaces admit K\"ahler structures compatible with the symplectic form. 

\begin{theorem}[K\"ahler model for patterns of nice symplectic surfaces] \label{thm:kahler-model-union}
Let $(X, \o)$ be a symplectic $4$-manifold, and
let $S_j \subset X$ be compact symplectic surfaces intersecting nicely. Consider $\cU\subset X$ 
a small tubular neighborhood of the union $\bigcup_j S_j$. Then there exists 
a symplectomorphism $\f:\cU \to \cV$, where $\cV$ is a K\"ahler open $4$-manifold, and $\f(S_j) \subset \cV$ are K\"ahler Riemann surfaces (complex curves).
\end{theorem}

\begin{proof}
Let $\o_j$ be the restriction of $\o$ to $S_j$. Note that $(S_j,\o_j)$ admits many complex structures compatible with $\o_j$, so we select one and call it $J_j$. Consider a holomorphic line bundle $L_j \to S_j$ with $\deg L_j=S_j^2$, e.g.\ a holomorphic bundle associated to a divisor of degree $S_j^2$, and the 
K\"ahler structure $(\o_{L_j}, J_{L_j})$ defined in 
Lemma \ref{lem:kahler-line-bundle}, on 
some small disc bundle $B_r(S_j) \subset L_j$.
By Proposition \ref{prop:tubular1},
there exists a symplectomorphism 
$\f_j: \cU_j \to B_r(S_j)$, where $\cU_j\subset X$ is a small tubular neighborhood of $S_j$ in $X$.

Moreover, let $p$ be any point of intersection in $S_j \cap S_k$. Note that the intersection is nice, so there are only two surfaces going through $p$, and Darboux coordinates $(u_1,u_2)$ on a 
ball $W$ around $p$ so that 
 $$
 S_j=\{u_2=0\}, \quad S_k=\{u_1=0\}. 
 $$
By Remark 
\ref{rem:rem3}, we have neighbourhoods $V_j$ and $V_k$ of
$p$ in $S_j$ and $S_k$, respectively, so that we can assume that the complex structures $J_j$, $J_k$ give both flat 
K\"ahler structures. Next, we choose hermitian metrics $h_j$ and $h_k$ so that the second item of Lemma \ref{lem:kahler-line-bundle} applies and hence we have diffeomorphisms
 $$
 L_j|_{V_j} \cap B_r(S_j) \cong V_j\x B_r(0), \quad
 L_k|_{V_k} \cap B_r(S_k) \cong V_k\x B_r(0) 
 $$
that preserve the K\"ahler structures,
where the K\"ahler structure on the right hand side is the 
product flat K\"ahler structure.
If we arrange $V_j$, $V_k$ small so that $V_j=B_r(0)$ and
$V_k=B_r(0)$ then
 $$
 L_j|_{V_j} \cap B_r(S_j) \cong B_r(0)\x B_r(0), \quad
 L_k|_{V_k} \cap B_r(S_k) \cong B_r(0)\x B_r(0) \, .
 $$
Using the second item of Proposition \ref{prop:tubular1}, 
we can arrange $\f_j:\cU_j\to B_r(S_j)$ so that, when restricted to $W$, it
coincides with the Darboux chart $W \cong B_r(0)\x B_r(0)$,
that is $\f_j(u_1,u_2)=(u_1,u_2)$ on $W$.
Analogously, we arrange $\f_k:\cU_k\to B_r(S_k)$
to coincide with the Darboux chart. As $S_k=\{u_1=0\}$,
we have $\f_k(u_1,u_2)=(u_2,u_1)$ on $W$.

In $B_r(0)\x B_r(0)$ we have a flat K\"ahler structure with coordinates $(z,w)$, so we can glue
 $$
 L_j|_{V_j} \cap B_r(S_j) \cong B_r(0)\x B_r(0) \to B_r(0)\x B_r(0) \cong
 L_k|_{V_k} \cap B_r(S_k) \, 
 $$
via $R(z,w)= (w,z)$. The resulting manifold 
$B_r(S_j) \cup_R B_r(S_k)$ has a well-defined K\"ahler structure. Moreover, $\f_j$ and $\f_k$ give a well-defined 
symplectomorphism $\f_{jk}:\cU_j\cup \cU_k \to 
B_r(S_j) \cup_R B_r(S_k)$, since on the intersection 
$\cU_j\cap \cU_k=W$ we have arranged $\f_j=R\circ \f_k$.
 
 We perform the above gluing for all the intersection points $p$ of any pair of surfaces $S_j \cap S_k$, for $j \ne k$, and the result is a K\"ahler 
 manifold $\cV=\bigcup_{R_p} B_r(S_j)$, which contains the collection of complex curves $S_j$ intersecting nicely. Clearly $\cV$ is symplectomorphic to 
 $\cU=\bigcup_j \cU_j$, 
 which is a neighbourhood of the union $\bigcup_j
 S_j\subset X$ of surfaces intersecting nicely in $X$.
\end{proof}

We can extend the previous result when we have a collection of surfaces intersecting in a complex-like manner
(see Definition \ref{def:complex-like}).

\begin{proposition} [K\"ahler model for patterns of complex-like symplectic surfaces] \label{prop:extra}
    Let $X$ be a symplectic $4$-manifold, and suppose that $S_i\subset X$, $i\in I$, are compact
    symplectic surfaces such that at each intersection
    point of several of the surfaces, say $S_0,\ldots,
    S_l$ , there are
    Darboux coordinates $(z,w)$ so that 
    $S_j=\{w=a_j z\}$, for $0\leq j\leq l$, $a_0,\ldots,
    a_l\in \CC$ distinct complex numbers.
     Consider $\cU$ a small tubular neighborhood of the union $\bigcup_j S_j$. Then there exists 
a symplectomorphism $\f:\cU \to \cV$, where $\cV$ is a K\"ahler open $4$-manifold, and $\f(S_j) \subset \cV$ are K\"ahler Riemann surfaces (complex curves).
\end{proposition}

\begin{proof}
    We work as in Theorem \ref{thm:kahler-model-union}. 
We start by taking a K\"ahler structrure $(\o_j,J_j)$ on each $S_j$ and a holomorphic line bundle $L_j\to S_j$ of
degree $\deg L_j=S_j^2$. By Proposition \ref{prop:tubular1},
there is a symplectomorphism $\f_j:\cU_j\to B_r(S_j)$, where
$\cU_j\subset X$ is  a tubular neighbourhood of $S_j$, and
$B_r(S_j)\subset L_j$ is the disc bundle of radius $r>0$ around
$S_j$ in $L_j$. Let $(u_j,v_j)$ denote the coordinates of $L_j$, where $u_j$ is the coordinate of $S_j$.
 
For each point $p\in S_0\cap S_1\cap \ldots \cap S_l$, with $l\geq 1$, we take Darboux coordinates
$\psi\colon W\to B_r^4(0)$, which we denote $(z,w)$, such that
$S_j=\{w=a_jz\}$, where $a_0,\ldots,a_l\in \CC$ are distinct
complex numbers. We give $B_r^4(0)$ the flat K\"ahler structure
given by $\o_0=\frac{\ii}2 (dz\wedge d\bar z+ dw
\wedge d\bar w)$.
Take the map $R_j:B_r^4(0)\to B_r^4(0)$, $(z,w)=R_j(u_j,v_j)$,
given by the matrix
$$
 \frac{1}{\sqrt{1+|a_j|^2}}\left(\begin{matrix} 1 & -\bar{a}_j \\
 a_j & 1\end{matrix} \right) \in \mathrm{SU}(2).
 $$
This is an isometry (of the flat K\"ahler structures) which 
sends $\{v_j=0\}$ to $\{w=a_jz\}$.

Next we arrange
a neighbourhood $V_j=B_r(0)$ around $p$ in $S_j$ (using Remark \ref{rem:rem3}) so that
the K\"ahler structure $(\o_j,J_j)$ is flat on $V_j$. By the 
second item of Proposition \ref{prop:tubular1}, we arrange
the symplectomorphism
    $$
    \f_j: \f_j^{-1}( L_j|_{V_j}\cap B_r(S_j)) \subset \cU_j
    \to
    L_j|_{V_j}\cap B_r(S_j)\cong B_r(0)\x B_r(0),
    $$
that sends $S_j$ to $B_r(0)\x \{0\}=\{(u_j,0)\}$, so that the restriction of $\f_j$
coincides with the Darboux chart $R_j^{-1} \circ \psi:W\to
B_r^4(0)$,
where the coordinates in the right hand side are $(u_j,v_j)$. 

We perform the gluing $\cV=\bigcup_R B_r(S_j)$, where
we take the equivalence relation $(u_j,v_j)\sim_R (u_k,v_k)$ if
$R_j(u_j,v_j)=R_k(u_k,v_k)$. 
The space $\cV$ inherits a K\"ahler structure since in the gluing part the K\"ahler structures are flat. 
We have a well-defined symplectomorphism
 $$
 \f: \cU={\textstyle \bigcup_j}\,  \cU_j\to  \cV=
 {\textstyle\bigcup_R}\, B_r(S_j),
$$ 
by gluing the maps $\f_j$, 
since $R_j\circ \f_j=R_k\circ \f_k$ 
on $\cU_j\cap \cU_k\subset W$ (and the 
latter inclusion holds taking $r>0$ small). 
\end{proof}

\section{Representing symplectic divisors} \label{section:symplectic-divisors}

Let $(X,\o)$ be a symplectic $4$-manifold. A symplectic divisor is defined as a finite
formal sum $\sum a_i C_i$ where $C_i \subset X$ are (smooth) compact symplectic surfaces, 
$a_i\in \ZZ$. Clearly this has a well-defined homology class.
 We study now the possibility of realizing the homology class of a symplectic divisor $\sum a_i C_i$ by a (smooth) symplectic surface.
We need some conditions:
\begin{itemize}
    \item The divisor must be effective, i.e. $a_i \ge 0$.
    \item We assume the surfaces $C_i$ to
    intersect transversely and positively (the self-intersections $C_i^2$ may be negative).
\end{itemize}
Let us start with a simple case: the divisor $nC$ when $C^2>0$.

\begin{proposition}\label{prop:Cn}
    Let $X$ be a symplectic $4$-manifold and $C \subset X$ a symplectic surface with $C^2 > 0$.
    Then for any positive integer $n$ there exists a symplectic surface $C_n$ whose homology class is $[C_n]=n[C]$. The construction can be done inside an arbitrary tubular
    neighbourhood of $C$.
    
Moreover, for $N\geq 1$, 
    the family of surfaces $C_n$, $1\leq n \leq N$, can be constructed so that
    they intersect nicely.
\end{proposition}

\begin{proof}
Using Lemma \ref{lem:kahler-line-bundle}, we can 
take a K\"ahler model $\cV=B_r(C)$ for a tubular neighborhood of $C$, i.e. a disc bundle of some 
holomorphic hermitian line bundle $L \to C$ with 
$d=\deg L=C^2>0$. We can choose $L=\cO_C(\sum_l p_l)$ associated to a divisor given by the sum of $d=C^2$ 
distinct points. Hence there exists a non-zero holomorphic section $\s:C \to L$, with zero set given by $\sum_{l=1}^d p_l$. 

For every $n\geq 1$, we take the sections $\s_{ni}=\l_{ni}\, \s$, $i=1,\dots , n$, where $\l_{ni} \ne 0$ are complex numbers so that
$\l_{ni} \ne \l_{mj}$, for all $n,m\geq 1$, 
$1\leq i\leq n$, $1\leq j\leq m$, $(n,i)\neq (m,j)$. We can take
$|\lambda_{ni}|$ small so that all graphs
$\G_{ni}=\G(\s_{ni})$ lie inside $\cV$. 
As the section $\s$ vanishes at the points $p_1, \dots , p_d$ with multiplicity one, the graphs 
$\G_{ni}$, for $1 \le i \le n$, 
intersect positively and transversely at the points $p_1, \dots, p_d$. 

For each $n\geq 1$, consider the union
$Z_n=\bigcup_{i=1}^n \G_{ni}$. 
By Proposition \ref{prop:transverse-positive-implies-nice}, 
we can make a $C^0$-small perturbation of the graphs $\G_{n1}, \dots , \G_{nn}$ so that they intersect nicely. 
We denote by $\tilde{\G}_{ni}$ the perturbed graphs, and consider the union $\tilde Z_n=\bigcup_{i=1}^n
\tilde{\G}_{ni}$, which has ordinary double points at the intersection of the graphs. Now we use Lemma \ref{lem:resolution} to 
do a resolution of the nice intersections, and we 
obtain a smooth symlectic surface $C_n$ whose homology class is $[C_n]=[\tilde Z_n]=[Z_n]=n [C]$.

To prove the second statement, take 
the whole collection of graphs $\G_{ni}$, for
$1\leq n \leq N$, $1\leq i \leq n$. There are 
 $1+2+ \ldots + N=\binom{N+1}{2}$ graphs, and they
 intersect transversely and positively. Applying
Proposition \ref{prop:transverse-positive-implies-nice},
we obtain the family of perturbed graphs $\tilde{\G}_{ni}$ intersecting nicely. Now
we do the process of symplectic resolution
of nice intersections for each $\tilde{Z}_n=
\bigcup_{i=1}^n \tilde{\G}_{ni}$, to obtain a symplectic
surface $C_n$, and all intersections $C_n\cap C_m$,
for $1\leq n\neq m \leq N$, are nice.
\end{proof}

It is relevant to compute the invariants of $C_n$ in 
Proposition \ref{prop:Cn}.

\begin{lemma}
In the situation of Proposition \ref{prop:Cn}, if
$C^2=k>0$ and $g(C)=g\geq 0$, then
$C_n$ has $C_n^2=n^2k$ and $g(C_n)=ng-n+1 +k\binom{n}{2}$.
\end{lemma}

\begin{proof}
As $[C_n]=n[C]$, clearly $C_n^2=n^2 C^2= n^2 k$.
The construction of $C_n$ is done by resolving the double points of $\tilde Z_n= \bigcup_{i=1}^n \tilde{\G}_{ni}$. The number of such double
points equals the number of intersections of $\tilde \G_{ni}$ and $\tilde \G_{nj}$, $i\neq j$. For $i, j$ fixed, we have $k=C^2$ points, so the total number of double points in $\tilde Z_n$
is $k\binom{n}{2}$. In order to compute the genus, recall that $C_n$ can be inductively constructed from $C_{n-1}$ as the resolution of $C_{n-1} \cup \tilde \G_{nn}$ at $(n-1)k$ double points, so we have the relation $g(C_n)=g(C_{n-1})+g(\tilde \G_{nn})+(n-1)k-1$. Applying this recursively we get
 $$
 g(C_n)= \sum_{i=1}^n g(\tilde{\G}_{in})+ k\binom{n}{2}
 -(n-1) =ng-n+1 +k\binom{n}{2}.
 $$
\end{proof}

\begin{remark}
Alternatively, the genus $g(C_n)$ can be computed by the (symplectic) adjunction formula: if $K_X$ is the canonical divisor of $X$, then $K_X \cdot C_n + C_n^2=2g(C_n)-2$. This is the method employed in \cite{Mu-jems}.
\end{remark}

These results are used in \cite[Proposition 3.2]{Mu-jems}
by starting with a symplectic surface $T_1$ with $k=T_1^2=18$ and
$g=g(T_1)=10$, to obtain $T_n=nT_1$ with $g(T_n)=10n-n+1
+18\binom{n}{2}=9n^2+1$.

\bigskip

Next we move on to more general divisors.

\begin{theorem}\label{thm:divisors}
Consider $C_i \subset X$ symplectic surfaces intersecting transversely and positively. Consider a
homology class 
$[D]=\sum_i a_i [C_i]$ with $a_i>0$ integers, and such that $[D] \cdot [C_j] \ge 0$, for all $j$. 

Then, inside a  
tubular neighborhood of $\bigcup_i C_i$, there exists a symplectic surface $\S$ representing the class $[D]$, and intersecting positively and transversely with all $C_j$. 
\end{theorem}
In particular, if $[D] \cdot [C_j]=0$ for some $j$, then $\S$ is disjoint from $C_j$.
\begin{proof}
We can assume that the $C_i$ intersect nicely by Proposition \ref{prop:transverse-positive-implies-nice}, 
and also that $X$ is K\"ahler by taking $X=\cV$ a K\"ahler 
model for a tubular neighborhood of the union $C=\bigcup_j C_j$, 
by Theorem \ref{thm:kahler-model-union}.
Consider the line bundle $L=\cO_X(D)$.
Let us call $n_j=[D]\cdot [C_j]  \ge 0$. 
Note that $C$ has singular points the points of intersection $C_j \cap C_k$, and these are ordinary double points. Select $n_j$ points in each $C_j$, distinct from the singular points of $C$, 
let us denote them $p_{j1}, \dots , p_{jn_j}$. We proceed in steps. 

\medskip

\noindent \textbf{Step 1}: We start by constructing a smooth section $s_C:C \to L|_C$ vanishing exactly at the points $p_{ji}$,
holomorphic in an open set containing the points 
$p_{ji}$, and constant in a neighbourhood of the double points.

Take some fixed curve $C_j$. The degree $\deg(L|_{C_j})=
[D]\cdot [C_j]=n_j$. We can assume that the points $p_{j1},
\dots , p_{jn_j}$ lie inside some $U_j \subset C_j$ which is biholomorphic to a disc $B_1(0) \subset \CC$ and that does not contain singular points. Consider some $r<1$ so that
the points are in $B_r(0)$, and the complement $V_j=C_j-B_r(0)$. The line bundles $L|_{U_j}$ 
and $L|_{V_j}$ are both (smoothly) trivial, and the change of trivialization is a degree $n_j$ map 
$q:U_j\cap V_j=B_1(0)-B_r(0)\to \CC^*$. We can take the trivialization $L|_{U_j} \cong U_j\x \CC$ holomorphic, 
with coordinates $(z,w)$. 
Take the local section $w=s(z)=\prod_{i=1}^{n_j} (z-p_{ji})$, which is a polynomial of degree $n_j$. 
In the trivialization $L|_{V_j}\cong V_j\x\CC$, we
have $\hat s(z)= q(z)^{-1} \prod_{i=1}^{n_j} (z-p^j_i)$, defined
on the intersection $A=U_j\cap V_j \cong B_1(0)-B_r(0)$. This is a degree $0$ map. 

On the other hand, if we contract the $1$-skeleton of $V_j$, then we have a continuous surjective map $\varpi:
V_j \to B_1(0)$, which sends the $1$-skeleton to $\{0\}$, 
and the annulus $A$ is sent by $\varpi$ to $B_1(0)-B_r(0)$ in a reversed way (inside-out), that is, with the map 
$\psi(z)=\frac{r}{z}$. We consider the map
$\hat{s}\circ \psi:B_1(0)-B_r(0)\to \CC^*$, which is of degree $0$, hence we can extend it to a map $g:B_1(0) \to \CC^*$, which
can be assumed to be $C^\infty$ by a small perturbation if necessary.
 Gluing the local section $s$ on $L|_{U_j}$ with $g \circ \varpi$ on $L|_{V_j}$, we get a 
smooth section $s_j:C_j\to L|_{C_j}$. 
Note that $s_{j}$ has zero set exactly the points $p_{ij}$,
$i=1,\ldots,n_j$. Moreover, $s_j$ is holomorphic on $U_j$.

In the construction above,
let $q_{j1},\ldots, q_{jr_j}\in V_j$ be the double
points that are in $C_j$. We can assume that they
are not in the $1$-skeleton, so the images $\varpi(q_{ji})$
are distinct points in $B_r(0)$. Therefore the extension $g(z)$
can be chosen so that the images $g(q_{ji})\in \CC^*$ take
any prefixed values. This means that we can prefix the values of
$s(q)$, for all double points $q$.

We define the section $s:C\to L|_C$ inductively. For $C_1$ it is the section $s_1$. Now define $s_2:C_2\to L|_{C_2}$ by requiring 
that $s_2(q)=s_1(q)$ for any $q\in C_1\cap C_2$. 
Therefore $s_1,s_2$ glue to a section $s_{12}:C_1\cup C_2 \to L|_{C_1\cup C_2}$. We iterate this process until we get 
a smooth section 
$$
s_C:C=C_1\cup \ldots\cup C_r \to L|_C=L|_{C_1\cup\ldots\cup C_r}\, .
$$
whose zero set consists precisely on the points $p_{ij}$, and $s_C$ is holomorphic in a neighborhood of its zero set.

\medskip

\noindent \textbf{Step 2}: Recall that $X=\cV$ is a tubular neighborhood of $C=\bigcup_j C_j$. We aim to see that there is a smooth retraction from $\cV$ to $C$, which coincides with the (holomorphic) normal bundle projection outside a neighborhood of the double points. Around each double point $q\in C_j\cap C_k$, we have Darboux coordinates $(z,w)$
with $C_j=\{w=0\}$, $C_k=\{z=0\}$. 
We can assume that
these coordinates are defined in the set where $|zw|\leq \eta^2$, $0 \le |z|, |w| \le 2\eta$, for $\eta>0$ small. We define the map:
 $$
 \phi_q(z,w)=\left\{ 
 \begin{array}{ll} \left(\frac{|z|-(|z||w|)^{1/2}}{2\eta-(2\eta|w|)^{1/2}}\,z,0\right), \qquad & |w| \le |z| \le 2\eta , \medskip \\
\left(0,\frac{|w|-(|z||w|)^{1/2}}{2\eta-(2\eta |z|)^{1/2}}\,w\right), 
\qquad & |z|\leq |w| \leq 2 \eta. \end{array}\right.
$$
This can be extended by the projection $(z,w) \to (z,0) \in C_j$ for 
$|z|\geq 2 \eta$, and by $(z,w) \to (0,w) \in C_k$ for $|w|\geq 2 \eta $. 
The map $\phi_q$ collapses, in a neighbourhood of $q$,
a tubular neighbourhood of $C_j\cup C_k$
onto $C_j\cup C_k$, and it can be made a smooth 
retraction by perturbing slightly. By 
taking holomorphic trivializations of the normal bundle, and
gluing all $\phi_q$, we get a smooth map 
$$
\pi: \cV \to C={\textstyle\bigcup_j}\,  C_j\, ,
$$
such that it is a holomorphic projection outside a small neighbourhood of the double points.
Now let 
 $$
 s=\pi^* s_C \in H^0(X,\pi^*L_C),
 $$
a section of the pull-back bundle $\pi^*L_C \to X$. This section vanishes along the discs $\pi^{-1}(p_{ji})$ 
over the $p_{ji}$ (recall that $X$ 
is a tubular neighborhood of $C$, identified with a disc bundle outside the double points). 
Note that $\pi^*L_C \cong L$ is an isomorphism as 
$C^{\infty}$-bundles, since $\pi$ is a deformation retraction with contractible fibers, so we can identify the section $s=\pi^* s_C$ with a smooth section of $L=\cO_X(D)$. Moreover, $s$ is holomorphic in a neighborhood of its zero set, since $s_C$ is holomorphic near its zero set $(s_C)^{-1}(0)=\{p_{ji}\}$, and $\pi$ is holomorphic in a neighborhood of the discs $\pi^{-1}(p_{ji} ) \subset X$.

\medskip

\noindent \textbf{Step 3}: Consider $\s \in H^0(X,L)$ the holomorphic section with $Z(\s)=D=\sum_i a_i C_i$, and let 
$\tilde \s_\epsilon=\s - \epsilon s$, a smooth section of $L$. 
We claim that the map 
\[
F: \CC^* \times X \to L, \qquad 
F(\epsilon,x) =  \s(x) - \epsilon \, s(x),
\]
is transverse to the zero section, where we take
$ \epsilon$ a complex number (which will be small later).
We must see that in any point $(\epsilon_0,x_0)$,
$\epsilon_0\neq 0$, where $F(\epsilon_0,x_0)=0$, then
$d_{(\epsilon_0,x_0)} F:\CC \times T_{x_0} X \to L_{x_0}$ is surjective. 
The map $F$ is holomorphic in the variable $\epsilon$, and $\frac{dF}{d\epsilon}(\epsilon_0,x_0) =-s(x_0)$, so if $s(x_0) \ne 0$ this gives surjectivity. It remains to see the points $x_0\in Z(s)$, which lie in the transversal discs at the points $p_{ji}\in C_j$.
At such a point, we have a chart $U\subset X$ with coordinates $(z,w)$,
where $C_j\cap U=\{(z,0)\}$, and $w$ is the vertical
coordinate in the
(holomorphic) tubular neighbourhood $X=\cV$. We
arrange so that $p_{ji}=(0,0)$, and  the projection $\pi(z,w)=z$ thus $x_0=(0,w_0)$. We take
a holomorphic trivialization $L|_U \cong U\x\CC$,
with coordinates $(z,w,u)$, so that $s(z,w)=z$.
Also $\s(z,w)=w^{a_j}$, since $Z(\s)=\sum a_j C_j$,
and $p_{ji}\in C_j$.
Therefore, in these coordinates, $F$ is the map
 $$
 u=F(\epsilon,z,w)= w^{a_j}-\epsilon \, z,
 $$
and we look at a point $(\epsilon_0,0,w_0)$. As
$F(\epsilon_0,0,w_0)=0$, we have $w_0=0$. Now we
look at $\frac{dF}{dz}=-\epsilon \neq 0$, hence the result.

\medskip

\noindent \textbf{Step 4}:
By parametric transversality, since $F$ is transversal to zero, we see that for a generic value of $\epsilon \in B_1(0)$ the map $\tilde \s_\epsilon$ 
is transverse to $0$, hence the set 
\[
\S=Z(\tilde \s_\epsilon)=\{x \in X\, | \, \s (x)=\epsilon\, s(x)\}
\]
is a smooth manifold. 
To see that $\S$ is a closed manifold, consider 
$Y=X - B_r(C)$, $0<r<1$, some neighborhood of the boundary of $X$. Since $\s^{-1}(0)=C$, $|\s| \ge K_1$ for some constant $K_1>0$ on $Y$. Let $K_0$ be an upper bound for $|s|$ on $X$. Then
\[
|\tilde \s_\epsilon(x)| \ge |\s(x)|-|\epsilon|\, |s(x)| 
\ge K_1-|\epsilon| K_0>0,  
\]
for $|\epsilon|<\frac{K_1}{K_0}$. For such generic 
$\epsilon\in \CC$, 
the zero set $Z(\tilde s_\epsilon)$ lies inside the open set $B_r(C) \subset X$, so it is closed submanifold of $X$.

On the other hand, $\S \cap C=Z(\s - \epsilon s) \cap C=\{p_{ji}\}$ because $\s$ vanishes on $C$ and $Z(s) \cap C=\{p_{ji}\}$. The intersections $\S \cap C$ are positive, because 
both $s$ and $\s$ are 
holomorphic near the $p_{ji}$, 
so $\tilde \s_\epsilon$ is also holomorphic near these points. Clearly, the homology class of $\S$ is 
\[
[\S]=[Z(\s - \epsilon s)]=[Z(\s)]=[D] \in H_2(X,\ZZ),
\]
as can be seen by letting $|\epsilon| \to 0$. 

\medskip

\noindent \textbf{Step 5:} It remains to see that $\S$ is symplectic if $\epsilon$ 
is small. For this we must prove that $|\bd \tilde \s_\epsilon(x)| 
> |\bar \bd \tilde \s_\epsilon(x)|$ 
at any point $x \in \S$. Here $|\cdot|$ denotes the norm with respect to a 
Hermitian metric on $L$, and $\bd \tilde \s_\epsilon$ 
and $\bar \bd \tilde \s_\epsilon$ are the $(1,0)$ and $(0,1)$ parts of $\nabla \tilde \s_\epsilon$ for any choice of connection. 
Note that as $\tilde \s_\epsilon(x)=0$, then $\bd \tilde \s_\epsilon(x)$ and $\bar \bd \tilde \s_\epsilon(x)$ are obtained by simply applying 
the operators to the local functions in a trivialization and there is no need to specify the connection. 

Take a point $x \in \S$, so $\s(x)=\epsilon s(x)$. 
We know that $s$ is holomorphic in some neighborhood of $s^{-1}(0)$, 
so there is some $r_1>0$ such that if $|s(x)|<r_1$ then $s$ is holomorphic in some neighborhood of $x$.
Thus if $|s(x)|<r_1$ then $\S$ is 
symplectic near $x$, being (locally) the zero set of a holomorphic function in a K\"ahler manifold.
So assume $|s(x)| \ge r_1$ from now on. 
Then $|\s(x)|=|\epsilon s(x)| \ge |\epsilon|\, r_1$
Let $(z,w)$ be coordinates for $X$ near the point $x$. 
We have two cases: 

\begin{itemize}
    \item 
If $x$ is away from a double point, then
$\s(z,w)=w^a$ for some $a \ge 1$, so $|w|^a \ge |\epsilon|\, r_1$, and $|w| \ge (r_1 |\epsilon|\,)^{\frac1a}$. Hence
\begin{align*}
    | \bd \tilde \s_\epsilon| &= |aw^{a-1} d w - \epsilon 
    \bd s| \ge a |w|^{a-1} 
    - |\epsilon|\, | \bd s| \ge a r_1^{\frac{a-1}{a}} |\epsilon|^{\frac{a-1}{a}}-|\epsilon|\, K_2 \\
    &= |\epsilon|^{\frac{a-1}{a}}( a r_1^{\frac{a-1}{a}}
    -|\epsilon|^{\frac{1}{a}} K_2) >  |\epsilon|\, K_2 \ge 
    |-\epsilon\, \bar \bd s| 
    =| \bar \bd \tilde \s_\epsilon|, 
\end{align*}
for $\epsilon$ small,  
where we denote $K_2>0$ a constant so that 
$|s|, |\bd s|, |\bar \bd s| \le K_2$ on $X$, 
and we have used that  $|\epsilon|^{\frac{a-1}{a}} \gg |\epsilon|$ 
as $|\epsilon| \to 0^+$.

\item 
If $x$ is near a double point, we can take
the coordinates $(z,w)$ centered at the double point. 
Thus $\s(z,w)=w^a z^b$, for some $a,b\geq 1$.
As $\tilde\s_\epsilon(z,w)=0$, we have 
$|\s(z,w)|=|w|^a |z|^b = |\epsilon|\, |s(z,w)|$, so we have 
\begin{align*}
\bd \s & =a w^{a-1} z^b dw+bw^az^{b-1} dz \,  , \\
r_1 |\epsilon|& \le |w|^a |z|^b \le K_2 |\epsilon|   \, .
\end{align*}
Here either 
$|z|^b \le (K_2 |\epsilon| )^{\frac12}$ or $|w|^a\le (K_2 |\epsilon| )^{\frac12}$. Both cases are similar, so we do the first one. If
$|z|^b \le (K_2 |\epsilon|)^{\frac12}$ then 
$$
|w|^{a} |z|^{b-1} =\frac{|w|^a |z|^b}{|z|} \ge \frac{r_1 |\epsilon|}{(K_2 \, |\epsilon| )^{\frac{1}{2b}}}
=K_3 \,|\epsilon|^{1-\frac{1}{2b}}\, ,
$$
for some $K_3>0$. Hence 
\[
|\bd \s| =|a \,w^{a-1} z^b dw+b\,w^az^{b-1} dz| \ge |b \,w^{a} z^{b-1}| \ge  b K_3\, |\epsilon|^{1-\frac{1}{2b}}\, ,
\]
where we have used a flat K\"ahler chart around the double point, so that $dw, dz$ are an orthonormal basis of $1$-forms. This can be done as in Theorem \ref{thm:kahler-model-union}. Alternatively, we can work adding some constant to control the norms of the $1$-forms.

With this, 
\[
|\bd \tilde \s_\epsilon| \ge |\bd \s|- |\epsilon|\, |\bd s| \ge bK_3  \, |\epsilon|^{1-\frac1{2b}} - |\epsilon|\, K_2 
> |\epsilon| \, K_2 \ge |-\epsilon\, \bar \bd s|= |\bar \bd \tilde \s_\epsilon|, 
\]
for $\epsilon$ small,
where we have used that $|\epsilon|^{1-\frac{1}{2b}} 
\gg |\epsilon|$, as $|\epsilon| \to 0^+$. This yields that $\S$ is a symplectic manifold and completes the proof.
\end{itemize}
\end{proof}

We have used this result in \cite[Proposition 3.2]{Mu-jems}, where we have a symplectic torus $T$ and a symplectic sphere $D$ with $T^2=0$, $D^2=-2$, $T \cdot D=3$, and we construct a symplectic surface $T_1$ representing the divisor $3D+2T$, and intersecting $D$ and $T$ transversely and positively. In particular, $T_1$ is disjoint from $D$. This fact is important later when $D$ is contracted to a singular point, since $T_1$ avoids the singularity.

We also use it in
\cite[Proposition 3.3]{Mu-jems} where we have a chain of
symplectic surfaces $F$, $E_1$, $C_1,\ldots, C_8$, $C_1',\ldots, C_8'$, 
and we consider the symplectic divisor
$$
A = 2F+ 9 E_1 + 8(C_1+C_1')+7(C_2+C_2')+\dots + (C_8+C_8')\, .
$$
Hence by Theorem \ref{thm:divisors}, 
we obtain a symplectic surface $\Sigma \subset X$, 
in a neighbourhood of the
collection of curves, representing the class $A$.
As $A\cdot E_1=0$, $A\cdot C_i=0$ and $A\cdot C_i'=0$, for all $i$,
the 
resulting $\Sigma$ only intersects $F$. This is used later in \cite{Mu-jems} since
all curves $E_1,C_1,\ldots, C_8, C_1',\ldots, C_8'$ are contracted to a singular 
point, and $\Sigma$ will avoid this singularity.
This method to obtain $\S$ is simpler than the ad-hoc argument that appears in the
proof of \cite[Proposition 3.3]{Mu-jems}.

\end{document}